\documentclass[reqno]{amsart}
\usepackage{amsmath}
\usepackage{graphicx, xcolor} 
\usepackage{subfiles}
\usepackage{enumitem}
\usepackage[T1]{fontenc}
\usepackage{dsfont}
\newtheorem{theorem}{Theorem}[section]
\newtheorem{lemma}{Lemma}[section]
\newtheorem{proposition}{Proposition}[section]

\newtheorem{notation}{Notation}[section]
\newtheorem{definition}{Definition}[section]
\theoremstyle{remark}
\newtheorem{remark}{Remark}[section]
\usepackage{physics}
\usepackage{mathtools}
\usepackage{amsmath, amssymb}
\def\undertilde#1{\mathord{\vtop{\ialign{##\crcr
$\hfil\displaystyle{#1}\hfil$\crcr\noalign{\kern1.5pt\nointerlineskip}
$\hfil\widetilde{}\hfil$\crcr\noalign{\kern1.5pt}}}}}
\setlist[enumerate,2]{label=\arabic*)} \newcommand*\diff{\mathop{}\!\mathrm{d}}
\newcommand{\R}{\mathbb{R}}
\newcommand{\eps}{\epsilon}

\newcommand{\V}{\mathcal{V}}
\allowdisplaybreaks

\usepackage{float}
\usepackage{svg}

\title[$L^2$-stability for shocks without Genuine Nonlinearity]{Relative Entropy Contractions for Extremal Shocks of Nonlinear Hyperbolic Systems without Genuine Nonlinearity}

\author{Jeffrey Cheng}
\address{Department of Mathematics, The University of Texas at Austin, Austin, TX 78712.}
\email{jeffrey.cheng@utexas.edu}

\date{\today}
\thanks{\textbf{Acknowledgements}: The author thanks his graduate advisor Alexis Vasseur for helpful conversations and advice regarding Section \ref{applications}. This work was partially supported by NSF grants: DMS-1840314 \& DMS-2306852 and a joint NSF-ESPRC grant: DMS-EPSRC-2219434}
\begin{document}

\begin{abstract}
We study extremal shocks of $1$-d hyperbolic systems of conservation laws which fail to be genuinely nonlinear. More specifically, we consider either $1$- or $n$-shocks in characteristic fields which are either concave-convex or convex-concave in the sense of LeFloch. We show that the theory of $a$-contraction can be applied to obtain $L^2$-stability up to shift for these shocks in a class of weak solutions to the conservation law whose shocks obey the Lax entropy condition. Our results apply in particular to the $2 \times 2$ system of nonlinear elastodynamics.
\end{abstract}

\maketitle
\tableofcontents

\section{Introduction \& main results}
We consider a $1$-d hyperbolic system of conservations laws
\begin{align}\label{hs}
u_t+(f(u))_x=0, \indent (t,x) \in \R^+ \times \R,
\end{align}
where $u \in \V \subset \R^n$ is the unknown. We assume that the state space $\V$ is compact and that $f=(f_1, ..., f_n) \in[C^2(\V)]^n$. We assume that the system $\eqref{hs}$ is strictly hyperbolic. That is, we assume that for any $u \in \V$, $f'$ has $n$ distinct real eigenvalues with ordering
\[
\lambda_1(u) < ... < \lambda_n(u).
\]
We denote by $r_1(u), ..., r_n(u)$ the associated right eigenvectors corresponding to the $i$-characteristic field for $i \in \{1, ..., n\}$. Recall that the $i$-characteristic field is called genuinely nonlinear if
\begin{align}\label{gnl}
\lambda_i'(u) \cdot r_i(u) \neq 0.
\end{align}
The majority of the theory for hyperbolic systems of conservation laws has been developed under the assumption that \eqref{gnl} (or the linear degeneracy condition) holds. In these contexts, existence for small-BV solutions has shown via multiple methods and very powerful uniqueness and stability results have been established (see the fundamental classical results in these regards \cite{MR194770}, \cite{MR1489317}, \cite{MR1757395}, \cite{MR2150387}, \cite{MR1723032}). For a comprehensive account see \cite{MR1816648}. In this paper, we will work in contexts where \eqref{gnl} does not hold. Although many important physical systems (compressible Euler, for example) verify \eqref{gnl}, there are many physical applications where \eqref{gnl} is not verified. These include modeling the motion of elastic solids \cite{MR1856989}, phase transitions \cite{MR683192}, \cite{MR1219421}, and heat propagation in crystals \cite{RMS1}. There are much fewer results in these directions, but existence of the semigroup of small-BV solutions has been shown by Ancona \& Marson \cite{MR2044745}, \cite{MR1876651} (see also \cite{MR603391}) and uniqueness was recently established by Bressan \& De Lellis \cite{MR4661213}.
\par We further assume that the system is endowed with a strictly convex entropy $\eta \in C^2(\V)$ which verifies
\begin{align}\label{entropy}
\eta'=q'f', 
\end{align}
for some entropy flux $q \in C^2(\V)$. We will consider solutions to $\eqref{hs}$ which verify the entropy condition
\begin{equation}\label{entropic}
(\eta(u))_t+(q(u))_x \leq 0 \indent t > 0, x \in \R,
\end{equation}
in the sense of distributions. More specifically, we ask that for all $\psi \in C_0^\infty(\R^+ \times \R)$ verifying $\psi \geq 0$:
\[
\int_0^\infty \int_{-\infty}^\infty \bigl(\psi_t(t,x)\eta(u(t,x))+\psi_x(t,x)q(u(t,x))\bigr)\diff x \diff t \geq 0.
\]
Given the entropy $\eta$ and entropy flux $q$, we may define the relative entropy and relative entropy flux as follows:
\begin{align}\label{relatives}
    \eta(u|v):=\eta(u)-\eta(v)-\eta'(v)(u-v), \\
    q(u;v):=q(u)-q(v)-q'(v)(f(u)-f(v)).
\end{align}
The relative entropy $\eta(\cdot|\cdot)$ defines a distance between states which is equivalent to the square of the $L^2$-distance, which is made precise in the following lemma:
\begin{lemma}[from \cite{MR2807139}, \cite{MR2508169}]\label{quadraticentropy}
There exists $c^*, c^{**} > 0$ such that for all $(u,v) \in \V \times \V$:
\begin{equation*}
    c^*|u-v|^2 \leq \eta(u|v)\leq c^{**}|u-v|^2.
\end{equation*}
\end{lemma}
We will also only consider solutions that verify the strong trace property:
\begin{definition}\label{strongtrace}
Let $u \in L^\infty(\R^+ \times \R)$. We say that $u$ verifies the strong trace property if for any Lipschitzian curve $t \mapsto h(t)$, there exist two bounded functions $u_-, u_+ \in L^\infty(\R^+)$ such that for any $T > 0$:
\begin{align*}
&\lim_{n\to \infty}\int_0^T\sup_{y \in (0,\frac{1}{n})}|u(t,h(t)+y)-u_+(t)|\diff t \\
&=\lim_{n\to \infty}\int_0^T\sup_{y \in (-\frac{1}{n},0)}|u(t,h(t)+y)-u_-(t)|\diff t=0.
\end{align*}
\end{definition}
\begin{notation}\label{notation}
For convenience, we will use later the notation $u_+(t)=u(t,h(t)+)$, and $u_-(t)=u(t,h(t)-)$. 
\end{notation}
\par A special class of solutions to $\eqref{hs}$ are the so-called shock waves $(u_L, u_R)$, which are functions of the form
\begin{align}\label{shock}
S(t,x)=\begin{cases}
u_L & x < \sigma t, \\
u_R & x > \sigma t,
\end{cases}
\end{align}
for some speed $\sigma \in \R$. Recall that an $i-$shock $(u_L, u_R)$ verifies the Lax entropy condition if
\begin{align}\label{lax}
\lambda_i(u_L) \geq \sigma \geq \lambda_i(u_R).
\end{align}
\begin{remark}
For the class of fluxes we are considering, a shock verifies the Lax condition \eqref{lax} if and only if it verifies the ostensibly stronger Liu condition (see Lemma 2.4 in \cite{MR1759195}).
\end{remark}
The goal of the paper is to show an $L^2$-stability property up to shift for Lax shocks in $1$- or $n$-characteristic fields that are not genuinely nonlinear. 

\subsection{Concave-convex and convex-concave characteristic fields}
In this paper, we consider either $1-$ or $n$-characteristic fields that are either concave-convex or convex-concave in the sense of LeFloch \cite{MR1927887}, \cite{MR1759195}. In particular these fields will not verify the genuine nonlinearity condition \eqref{gnl}. More specifically, consider the scalar field
\begin{align*}
m_i(u):=\lambda_i'(u)\cdot r_i(u).
\end{align*}
Then, we have the following definition:
\begin{definition}\label{concave-convex}
The characteristic field $\lambda_i(u)$ is called concave-convex (convex-concave) if the following conditions hold.
\begin{enumerate}
\item $m_i'(u) \cdot r_i(u) > 0$ $(< 0)$. 
\item The set $\mathcal{M}_i:=\{u \in \V|m_i(u)=0\}$ is a smooth affine manifold of dimension $n-1$ such that the vector field $r_i$ is transverse to $\mathcal{M}_i$.
\end{enumerate}
\end{definition}
Examples of systems with characteristic fields of these types include scalar conservation laws with concave-convex/convex-concave fluxes and the equations of nonlinear elastodynamics (see Sections \ref{scl} and \ref{elastodynamics} respectively). 

\subsection{Main results}
Our main theorems give $L^2$-stability for $1$- or $n$-shocks in characteristic fields verifying Definition \ref{concave-convex} in a class of solutions $\mathcal{S}_{weak}$, which is defined as
\begin{align}
\mathcal{S}_{weak}:=\{u \in L^\infty(\R^+ \times \R;\V)|\text{weak solution to } \eqref{hs}\eqref{entropic} \text{ verifying Def. \ref{strongtrace}} \\ \notag
\text{whose shocks verify \eqref{lax} (the Lax entropy condition)}\}
\end{align}
Let us first explain what we can prove for the system of nonlinear elastodynamics:
\begin{align}\label{elastsystem}
\begin{cases}
w_t-v_x=0, & (t,x) \in \R^+ \times \R, \\
v_t+p(w)_x=0.
\end{cases}
\end{align}
The functions $v$ and $w$ represent the velocity and deformation gradient respectively. The stress-strain law $p$ is assumed to have the form:
\begin{align}\label{stresslaw}
p(w)=-w^3-mw, \indent m>0.
\end{align}
If we define $u=(w,v)^T$, the system has strictly convex entropy
\begin{align}
\eta(u)=\frac{v^2}{2}+\frac{1}{4}w^4+\frac{mw^2}{2}.
\end{align}
The $1$-shock curve emanating from a state $u_L$ and the $2$-shock curve emanating from a state $u_R$ are as follows:
\begin{align*}
&S^1_{u_L}(s)=(w_L+s,v) \text{ such that } v=v_L+\sqrt{\frac{-(p(s+w_L)-p(w_L))}{s}}(s), \\
&S^2_{u_R}(s)=(w_R+s,v) \text{ such that } v=v_R-\sqrt{\frac{-(p(s+w_R)-p(w_R))}{s}}(s). \\
\end{align*}
See Section \ref{elastodynamics} for more details about the system \eqref{elastsystem}. What we can prove is the following.
\begin{theorem}\label{main4}
There exists $\eps > 0$ and $0 < a < 1$ $(1<a<\infty)$ such that the following holds. Let $u_L=(w_L, v_L)$ be such that $w_L \neq 0$ ($u_R$ be such that $w_R \neq 0$) and $u_R=S^1_{u_L}(s_0)$ for some $0 < s_0 < \eps$ ($u_L=S^2_{u_R}(s_0)$ for some $0 < s_0 < \eps$). For any $u \in \mathcal{S}_{weak}$, there exists a Lipschitz function $h(t)$ with $h(0)=0$ such that the following dissipation functional verifies:
\begin{align*}
\dot{h}(t)\eta(u_-(t)|u_L)-q(u_-(t);u_L)-a(\dot{h}(t)\eta(u_+(t)|u_R)-q(u_+(t);u_R)) \leq 0,
\end{align*}
for almost every $t \in [0,\infty)$. In particular, there exists $C > 0$ such that
\begin{align*}
||u(t,\cdot+h(t))-S||_{L^2(\R)} \leq C||u_0-S||_{L^2(\R)},
\end{align*}
where $S(x)=u_L$ for $x < 0$ and $S(x)=u_R$ for $x > 0$.
\end{theorem}
Theorem \ref{main4} is a simple consequence of the more general Theorem \ref{main2}. We note that the parameter $\eps$ in Theorem \ref{main4} may be large; for this reason, we say that Theorem \ref{main4} gives $L^2$-stability for shocks of ``moderate'' strength. See Section \ref{elastodynamics} for a comparison of our result with the $L^1$-stability theory of Ancona \& Marson \cite{MR1876651}.
\par In order to state our most general results, we now give a brief description of the properties of the shock curves we assume. A more precise list of our hypotheses may be found in Section \ref{assumptions}. The regularity we require on the wave curves and related parameters has been shown locally for general systems with characteristic fields verifying Definition \ref{concave-convex} by LeFloch (see Chapter 6, Section 2 in \cite{MR1927887}). Thus, although they may seem lengthy, our hypotheses are very natural for the fluxes we consider. All our assumptions are verified globally for the examples we consider in Section \ref{applications}. The assumptions are marked with asterisks to denote that they are abbreviated versions of the lists found in Section \ref{assumptions}.
\begin{enumerate}
    \item [(A1*)] Assume that the $1-$characteristic field is concave-convex. Then, for any $u_L \not \in \mathcal{M}_1$, there exists a neighborhood $B \subset \V \cap \mathcal{M}_1^c$ such that for any $u \in B$, the $1$-shock curve $S^1_u(s) \in \V$ is defined for $ s \in [0, s_u]$ for some $s_u \leq \infty$. That is, there exists a velocity function $\sigma^1_u(s)$ such that
    \begin{align*}
    f(S^1_u(s))-f(u)=\sigma^1_u(s)(S^1_u(s)-u),
    \end{align*}
    and the shock $(u, S^1_u(s))$ verifies the Lax condition \eqref{lax}. 
    \item [(A2*)] Assume that the $n-$characteristic field is convex-concave. Then, for any $u_R \not \in \mathcal{M}_n$, there exists a neighborhood $B \subset \V \cap \mathcal{M}_n^c$ such that for any $u \in B$, the $n$-shock curve $S^n_u(s) \in \V$ is defined for $ s \in [0, s_u]$ for some $s_u \leq \infty$. That is, there exists a velocity function $\sigma^n_u(s)$ such that
    \begin{align*}
    f(S^n_u(s))-f(u)=\sigma^n_u(s)(S^n_u(s)-u),
    \end{align*}
    and the shock $(S^n_u(s), u))$ verifies the Lax condition \eqref{lax}. 
    \item [(B1*)] Assume that the $1$-characteristic field is convex-concave. Then, for any $u_L \not \in \mathcal{M}_1$, there exists a neighborhood $B \subset \V \cap \mathcal{M}_1^c$ such that for any $u \in B$, the $1$-shock curve $S^1_u(s) \in \V$ is defined for $ s \in [-\infty, \infty]$. That is, there exists a velocity function $\sigma^1_u(s)$ such that
    \begin{align*}
    f(S^1_u(s))-f(u)=\sigma^1_u(s)(S^1_u(s)-u).
    \end{align*}
    Further, there exists a parameter $u \mapsto s_u^{-\natural} < 0$ such that $(u, S^1_u(s))$ verifies the Lax condition \eqref{lax} if and only if $s \in (-\infty, s_u^{-\natural}] \cup [0, \infty)$.
    \item [(B2*)] Assume that the $n$-characteristic field is concave-convex. Then, for any $u_R \not \in \mathcal{M}_n$, there exists a neighborhood $B \subset \V \cap \mathcal{M}_n^c$ such that for any $u \in B$, the $n$-shock curve $S^n_u(s) \in \V$ is defined for $ s \in [-\infty, \infty]$. That is, there exists a velocity function $\sigma^n_u(s)$ such that
    \begin{align*}
    f(S^n_u(s))-f(u)=\sigma^n_u(s)(S^n_u(s)-u).
    \end{align*}
    Further, there exists a parameter $u \mapsto s_u^{-\natural} < 0$ such that $(S^n_u(s), u)$ verifies the Lax condition \eqref{lax} if and only if $s \in (-\infty, s_u^{-\natural}] \cup [0, \infty)$.
    
\end{enumerate}
Now, we can state our first main theorem, which gives a relative entropy contraction (and $L^2$-stability up to shift as a simple corollary) for $1$-shocks in concave-convex fields or $n$-shocks in convex-concave fields.
\begin{theorem}\label{main1}
Consider $\eqref{hs}$ such that $\lambda_1$ is concave-convex and verifies assumptions (A1) ($\lambda_n$ is convex-concave and verifies assumptions (A2)). Let $u_L \in \V \cap \mathcal{M}_1^c$, $u_R=S^1_{u_L}(s_0)$ for some $0 < s_0 < s_{u_L}$ ($u_R \in \V \cap \mathcal{M}_n^c, u_L=S^n_{u_R}(s_0)$ for some $0 < s_0 < s_{u_R}$). Then, there exists $0 < a < 1$ ($1 < a < \infty$) such that the following holds. For any $u \in \mathcal{S}_{weak}$, there exists a Lipschitz function $h(t)$ with $h(0)=0$ such that the following dissipation functional verifies:
\begin{align*}
\dot{h}(t)\eta(u_-(t)|u_L)-q(u_-(t);u_L)-a(\dot{h}(t)\eta(u_+(t)|u_R)-q(u_+(t);u_R)) \leq 0,
\end{align*}
for almost every $t \in [0,\infty)$. In particular, there exists $C > 0$ such that
\begin{align*}
||u(t,\cdot+h(t))-S||_{L^2(\R)} \leq C||u_0-S||_{L^2(\R)},
\end{align*}
where $S(x)=u_L$ for $x < 0$ and $S(x)=u_R$ for $x > 0$.
\end{theorem}
Our second main theorem gives $L^2$-stability up to shift for $1$-shocks in convex-concave fields or $n$-shocks in concave-convex fields, conditional to a restriction that the shock is of ``moderate'' strength.
\begin{theorem}\label{main2}
Consider $\eqref{hs}$ such that $\lambda_1$ is convex-concave and verifies Assumptions (B1) ($\lambda_n$ is concave-convex and verifies Assumptions (B2)). Then, there exists $\eps > 0$ and $0 < a < 1$ ($1<a<\infty$) such that the following holds. Let $u_L \in \V \cap \mathcal{M}_1^c$, $u_R=S^1_{u_L}(s_0)$ for some $0 < s_0 < \eps$ ($u_R \in \V \cap \mathcal{M}_n^c, u_L=S^n_{u_R}(s_0)$ for some $0 < s_0 < \eps$). For any $u \in \mathcal{S}_{weak}$, there exists a Lipschitz function $h(t)$ with $h(0)=0$ such that the following dissipation functional verifies:
\begin{align*}
\dot{h}(t)\eta(u_-(t)|u_L)-q(u_-(t);u_L)-a(\dot{h}(t)\eta(u_+(t)|u_R)-q(u_+(t);u_R)) \leq 0,
\end{align*}
for almost every $t \in [0,\infty)$. In particular, there exists $C > 0$ such that
\begin{align*}
||u(t,\cdot+h(t))-S||_{L^2(\R)} \leq C||u_0-S||_{L^2(\R)},
\end{align*}
where $S(x)=u_L$ for $x < 0$ and $S(x)=u_R$ for $x > 0$.
\end{theorem}
\begin{remark}
Note that in Theorem \ref{main1} we cannot show stability for the critical case $s=s_{u_L}$. 
\end{remark}
\begin{remark}
The constant $\eps$ in Theorem \ref{main2} is explicit (see the proof of Lemma \ref{dissipationformula2}) and may be large in some cases, which is why we refer to shocks of ``moderate'' strength rather than ``small'' shocks.
\end{remark}
\begin{remark}
In the case of scalar conservation laws, we will be able to remove the moderate strength assumption from Theorem \ref{main2} (see Section \ref{scl}).
\end{remark}
\begin{remark}
Due to Lemma \ref{quadraticentropy}, the $L^2$-stability is a simple consequence of the contraction of the dissipation functionals.
\end{remark}
Theorems \ref{main1} and \ref{main2} are another advancement in Vasseur's theory of $L^2$-stability up to shift ($a$-contraction). The study of $L^2$-stability for hyperbolic systems via the relative entropy method was initiated by Dafermos \cite{MR546634} and DiPerna \cite{MR523630} for Lipschitz solutions. We note that the theory of Dafermos/DiPerna is extremely general and holds for multi-d systems with no assumptions of genuine nonlinearity on the flux. For $1$-d systems verifying \eqref{gnl}, the theory was extend to shocks by Leger \cite{MR2771666} and Kang \& Vasseur \cite{MR3519973}, \cite{MR2508169}. Recently, the theory has been extended even further to include piecewise smooth solutions \cite{MR4176349}, Riemann solutions \cite{MR4184662}, and BV solutions \cite{MR4487515}, \cite{MR4667839}, \cite{isothermal}.
\par This work marks another step towards recovering the generality of the result of Dafermos/DiPerna for discontinuous solutions. It was noticed by the author that the $a$-contraction theory applies in the scalar setting when the flux is concave-convex \cite{concaveconvex}. This work generalizes the ideas in \cite{concaveconvex} to the system setting for the class of fluxes in Definition \ref{concave-convex} introduced by LeFloch \& Hayes \cite{MR1759195}.
\par For scalar conservation laws, if the flux is not convex, it is well-known that a single entropy inequality is not sufficient to determine unique solutions to the Riemann problem. In fact, the so-called non-classical shocks (shocks that fail to verify \eqref{lax}) are also admissible and can be obtained as limits of diffusive-dispersive approximations, a fact which was first noted by Jacobs, McKinney, \& Shearer \cite{MR1318583}.
\par Selection principles that include non-classical shocks have been developed by Abeyaratne \& Knowles \cite{MR1094433}, LeFloch et al. \cite{MR1759195}, \cite{MR1475777}, \cite{MR1219421}, Slemrod \cite{MR683192}, and Shearer et al. \cite{MR1718538}, \cite{MR1357378} in both scalar and system settings. The corresponding undercompressive and overcompressive traveling waves have also been extensively studied \cite{MR1362168}, \cite{MR1324394}, \cite{MR1213739}, \cite{MR1240338}. Nonlinear stability has also been shown for the compressive waves \cite{MR4887752}, \cite{MR1303220}. 
\par In this paper, we will not focus on the non-classical shocks. Instead, our aim is to show that classical shocks possess strong stability properties that the non-classical ones may lack.
\par The article is structured as follows. Firstly, in Section \ref{gen} we state the precise assumptions we make on the system, explain the general ideas of the method of $a$-contraction and state our main propositions, Propositions \ref{mainprop1} \& \ref{mainprop2}. In Section \ref{theorempf}, we will prove Theorems \ref{main1} and \ref{main2} assuming the main propositions. In Section \ref{proppf}, we will prove the main propositions. Finally, in Section \ref{applications}, we will discuss applications of the theory to some specific systems.

\section{General framework}\label{gen}
\subsection{Assumptions on the system}\label{assumptions}
We now give the precise assumptions we make on the $1$- or $n$-characteristic fields. The first set of assumptions concerns either concave-convex $1$-characteristic fields or convex-concave $n$-characteristic fields.
\newline \textbf{Assumptions (A1,2)}
\begin{enumerate}[label=(\alph*)]
    \item (for $1$-shocks) The field $\lambda_1$ is concave-convex.
    \item (for $n$-shocks) The field $\lambda_n$ is convex-concave.
    \item (for $1$-shocks) For any $u_L \not \in \mathcal{M}_1$, there exists a neighborhood $B \subset \V \cap \mathcal{M}_1^c$ such that for any $u \in B$, the $1$-shock curve $S^1_u(s) \in \V$ is defined for $ s \in [0, s_u]$ for some $s_u \leq \infty$. That is, there exists a velocity function $\sigma^1_u(s)$ such that
    \begin{align*}
    f(S^1_u(s))-f(u)=\sigma^1_u(s)(S^1_u(s)-u),
    \end{align*}
    and the shock $(u, S^1_u(s))$ verifies the Lax condition \eqref{lax}. 
    \item (for $n$-shocks) For any $u_R \not \in \mathcal{M}_n$, there exists a neighborhood $B \subset \V \cap \mathcal{M}_n^c$ such that for any $u \in B$, the $n$-shock curve $S^n_u(s) \in \V$ is defined for $ s \in [0, s_u]$ for some $s_u \leq \infty$. That is, there exists a velocity function $\sigma^n_u(s)$ such that
    \begin{align*}
    f(S^n_u(s))-f(u)=\sigma^n_u(s)(S^n_u(s)-u),
    \end{align*}
    and the shock $(S^n_u(s), u))$ verifies the Lax condition \eqref{lax}. 
    \item (for $1$-shocks) The following properties are verified:
        \begin{enumerate}
            \item $S^1_u(0)=u$, $\sigma^1_u(0)=\lambda_1(u)$.
            \item $\dv{}{s}\sigma^1_u(s) < 0$ for $0 \leq s < s_u$.
            \item The functions $(s,u) \mapsto S^1_u(s), \sigma^1_u(s)$ are $C^1$ on $\{(s,u)|u \in B, 0 \leq s \leq s_u\}$ and the function $u \mapsto s_u$ is $C^1$ on $B$
        \end{enumerate}
    \item (for $n-$shocks) The following properties are verified:
        \begin{enumerate}
            \item $S^n_u(0)=u$, $\sigma^n_u(0)=\lambda_n(u)$.
            \item $\dv{}{s}\sigma^n_u(s) > 0$ for $0 \leq s < s_u$. 
            \item The functions $(s,u) \mapsto S^n_u(s), \sigma^n_u(s)$  are $C^1$ on $\{(s,u)|u \in B, 0 \leq s \leq s_u\}$ and the function $u \mapsto s_u$ is $C^1$ on $B$.
        \end{enumerate}
    \item (for $1$-shocks) $\dv{}{s}\eta(u|S^1_u(s)) > 0$.
    \item (for $n$-shocks) $\dv{}{s}\eta(u|S^n_u(s)) > 0$.
    \item (for $1$-shocks) If $(u,v)$ is a shock with velocity $\sigma$ that verifies the Lax condition with velocity $\sigma$ such that $u \in B$ and $\sigma \leq \lambda_1(u)$, then $v=S^1_u(s)$ for some $0 \leq s \leq s_u$.
    \item (for $n$-shocks) If $(v,u)$ is a shock with velocity $\sigma$ that verifies the Lax condition with velocity $\sigma$ such that $u \in B$ and $\sigma \geq \lambda_n(v)$, then $v=S^n_u(s)$ for some $0 \leq s \leq s_u$.
\end{enumerate}
The next set of assumptions concern either $1$-shocks in convex-concave characteristic fields of $n$-shocks in concave-convex characteristic fields. 
\newline \textbf{Assumptions (B1,2)}
\begin{enumerate}[label=(\alph*)]
    \item (for $1$-shocks) The field $\lambda_1$ is convex-concave.
    \item (for $n$-shocks) The field $\lambda_n$ is concave-convex.
    \item (for $1$-shocks) For any $u_L \not \in \mathcal{M}_1$, there exists a neighborhood $B \subset \V \cap \mathcal{M}_1^c$ such that for any $u \in B$, the $1$-shock curve $S^1_u(s) \in \V$ is defined for $ s \in (-\infty, \infty)$. That is, there exists a velocity function $\sigma^1_u(s)$ such that
    \begin{align*}
    f(S^1_u(s))-f(u)=\sigma^1_u(s)(S^1_u(s)-u).
    \end{align*}
    Further, there exists a parameter $u \mapsto s^{-\natural}_u < 0$ such that $(u, S^1_u(s))$ verifies the Lax condition \eqref{lax} if and only if $s \in (-\infty, s_u^{-\natural}] \cup [0, \infty)$. 
    \item (for $n$-shocks) For any $u_R \not \in \mathcal{M}_n$, there exists a neighborhood $B \subset \V \cap \mathcal{M}_n^c$ such that for any $u \in B$, the $n$-shock curve $S^n_u(s) \in \V$ is defined for $ s \in (-\infty, \infty)$. That is, there exists a velocity function $\sigma^n_u(s)$ such that
    \begin{align*}
    f(S^n_u(s))-f(u)=\sigma^n_u(s)(S^n_u(s)-u).
    \end{align*}
    Further, there exists a parameter $u \mapsto s^{-\natural}_u < 0$ such that $(u, S^1_u(s))$ verifies the Lax condition \eqref{lax} if and only if $s \in (-\infty, s_u^{-\natural}] \cup [0, \infty)$.
    \item (for $1$-shocks) The following properties are verified:
        \begin{enumerate}
            \item $S^1_u(0)=u$, $\sigma^1_u(0)=\lambda_1(u)$.
            \item There exists a parameter $s_u^\natural < 0$ such that $\dv{}{s}\sigma^1_u(s) < 0$ for $s \in (s_u^\natural, \infty)$  and $\dv{}{s}\sigma^1_u(s) > 0$ for $s \in (-\infty, s_u^\natural)$. 
            \item The functions $(s,u) \mapsto S^1_u(s), \sigma^1_u(s)$ are $C^1$ on $\{(s,u)|u \in B, s \in \R\}$. and the functions $u \mapsto s_u^{-\natural}, s_u^\natural$ are $C^1$ on $B$.
        \end{enumerate}
    \item (for $n-$shocks) The following properties are verified:
        \begin{enumerate}
            \item $S^n_u(0)=u$, $\sigma^n_u(0)=\lambda_n(u)$.
            \item There exists a parameter $s_u^\natural < 0$ such that $\dv{}{s}\sigma^n_u(s) > 0$ for $s \in (s_u^\natural, \infty)$  and $\dv{}{s}\sigma^n_u(s) < 0$ for $s \in (-\infty, s_u^\natural)$. 
            \item The functions $(s,u) \mapsto S^n_u(s), \sigma^n_u(s)$ are $C^1$ on $\{(s,u)|u \in B, s \in \R \}$ and the functions $u \mapsto s_u^{-\natural}, s_u^\natural$ are $C^1$ on $B$.
        \end{enumerate}
    \item (for $1$-shocks) $\dv{}{s}\eta(u|S^1_u(s)) > 0$ for $s > 0$, $\dv{}{s}\eta(u|S^1_u(s)) < 0$ for $s < 0$.
    \item (for $n$-shocks) $\dv{}{s}\eta(u|S^n_u(s)) > 0$ for $s > 0$, $\dv{}{s}\eta(u|S^n_u(s)) < 0$ for $s < 0$.
    \item (for $1$-shocks) If $(u,v)$ is a shock with velocity $\sigma$ that verifies the Lax condition with velocity $\sigma$ such that $u \in B$ and $\sigma \leq \lambda_1(u)$, then $v=S^1_u(s)$ for some $s \in (-\infty, s_u^{-\natural}] \cup [0, \infty)$.
    \item (for $n$-shocks) If $(v,u)$ is a shock with velocity $\sigma$ that verifies the Lax condition with velocity $\sigma$ such that $u \in B$ and $\sigma \geq \lambda_n(v)$, then $v=S^n_u(s)$ for some $s \in (-\infty, s_u^{-\natural}] \cup [0, \infty)$.
\end{enumerate}
\begin{remark}
Note that the characteristic field $\lambda_1$ verifies the Assumptions (A1) relative to $u_L$ if and only if the characteristic field $\lambda_n$ 
for the system
\begin{align*}
u_t-(f(u))_x=0, 
\end{align*}
verifies the Assumptions (A2) relative to $u_R$ (the same holds for (B1) and (B2)). It is enough to show Theorems \ref{main1} and \ref{main2} for $1-$shocks, as the result for $n$-shocks is obtained by applying the case of $1$-shocks to $\tilde{u}(t,x):=u(t,-x)$. Thus, we will henceforth only treat the cases of $1$-shocks (Assumptions (A1) and (B1)).
\end{remark}
\begin{remark}
The regularity we require on the shock curves and related parameters (Assumptions (A1) (ce) and similarly for the other cases) has been shown for general systems with characteristic fields verifying Definition \ref{concave-convex} by LeFloch (see Chapter 6, Section 2 in \cite{MR1927887} and Lemmas 2.3, 2.4 in \cite{MR1759195}). All our assumptions are verified globally for the examples we consider in Section \ref{applications}.
\end{remark}
\begin{remark}
In the case of Assumption (B1,2), the parameters $s_u^\natural, s_u^{-\natural}$ may not actually exist; the shock speed may be globally decreasing (increasing). The theory still applies in this case in a trivial manner, as then only the shocks with $s \geq 0$ are admissible and the theory is identical to the genuinely nonlinear setting. For ease of exposition, we will not consider this scenario at all.
\end{remark}
\begin{remark}
The remainder of the Assumptions (A1) (gi) are standard for the $a$-contraction theory \cite{MR3537479}, \cite{MR4667839}. The condition (i) has been studied by Zumbrun et al. \cite{MR3283661}, \cite{MR3338447}, who showed that if it is replaced by the seemingly similar condition $\dv{}{s}\eta(S^1_u(s)) > 0$, inviscid instability may occur even in the genuinely nonlinear setting. 
\end{remark}

\subsection{The method of $a$-contraction with shifts}
Let us briefly outline the method of $a$-contraction. Fix a shock $(u_L, u_R)$ with $u_R=S^1_{u_L}(s_0)$. The first step is to notice that an entropy-entropy flux pair $(\eta, q)$, actually gives rise to a full family of entropy-entropy flux pairs. Indeed, for any $b \in \V$, one can verify:
\begin{align}\label{relei}
(\eta(u|b))_t+(q(u;b))_x \leq 0,
\end{align}
in the sense of distribution. Then, for any solution $u \in \mathcal{S}_{weak}$, we define:
\begin{align}\label{functional}
\mathcal{E}(t)=\int_{-\infty}^{h(t)}\eta(u(t,x)|u_L) \diff x+a\int_{h(t)}^\infty \eta(u(t,x)|u_R) \diff x.
\end{align}
Using $\eqref{relei}$ and the strong trace property, we compute
\begin{align}\label{functionalderiv}
 \dv{}{t}\mathcal{E}(t) \leq \dot{h}(t)\eta(u_-(t)|u_L)-q(u_-(t);u_L)-a(\dot{h}(t)\eta(u_+(t)|u_R)-q(u_+(t);u_R)), 
\end{align}
where we have used the notation $u_-, u_+$ from Notation \ref{notation}. Thus, we see that negativity of the right-hand side of \eqref{functionalderiv} will immediately imply $L^2$-stability.
\par Let us analyze \eqref{functionalderiv} some more. We desire to show that the right-hand side is non-positive for some shift $h$. An important subset of state space will be:
\begin{align}
\Pi_a:=\{u \in \V|\eta(u|u_L)-a\eta(u|u_R)\leq 0\}.
\end{align}
Note that $\Pi_a$ is a convex and compact subset of $\V$ (cf. Lemma 3.5 in \cite{MR4667839}). If the weak solution $u$ does not have any discontinuity at $h(t)$ and $u(t,h(t)) \in \partial \Pi_a$, then the term containing $h(t)$ in $\eqref{functionalderiv}$ vanishes. Thus, in order to have the non-positivity, a necessary condition is
\begin{align}\label{dcont}
D_{cont}(u):=a(q(u;u_R)-\lambda_1(u)\eta(v|u_R))-(q(u;u_L)-\lambda_1(u)\eta(u|u_L)) \leq 0.
\end{align}
Our first main proposition shows that for $a$ sufficiently small, $D_{cont} \leq 0$ on $\Pi_a$ (not just the boundary).
\begin{proposition}\label{mainprop1}
There exists $\eps > 0$ such that the following is true. Consider \eqref{hs} such that the characteristic field $\lambda_1$ verifies Assumptions (A1) or (B1). Let $u_L \not \in \mathcal{M}_1$, $u_R=S^1_{u_L}(s_0)$ for some $s_0 > 0$. For assumptions (A1), assume that $s_0 < s_{u_L}$ and for assumptions (B1), assume that $s_0 < \eps$. Then, there exists $a^*>0$ such that for any $a < a^*$, and any $u \in \Pi_a$
\begin{align*}
    D_{cont}(u) \leq 0.
\end{align*}
\end{proposition}
Now consider the case when $u$ has a discontinuity at $h(t)$, that is $u_-(t) \neq u_+(t)$. We can show that in this case, $(u_-(t),u_+(t))$ corresponds to a shock $(u_-,u_+)$ with speed $\sigma_{\pm}=\dot{h}(t)$. By the definition of $\mathcal{S}_{weak}$, it must verify the Lax condition \eqref{lax}. The next proposition is a key step in ensuring that the right-hand side of \eqref{functionalderiv} is non-positive in these cases as well. For any such shock, define
\begin{align}\label{drh}
D_{RH}(u_-, u_+):=a(q(u_+;u_R)-\sigma_\pm \eta(u_+|u_R))-(q(u_-;u_L)-\sigma_{\pm}\eta(u_-|u_L)).
\end{align}
Then, we have
\begin{proposition}\label{mainprop2}
There exists $\eps > 0$ such that the following is true. Consider \eqref{hs} such that the characteristic field verifies Assumptions (A1) or (B1). Let $u_L \not \in \mathcal{M}_1$, $u_R=S_{u_L}^1(s_0)$ for some $s_0 > 0$. For assumptions (A1), assume that $s_0 < s_{u_L}$ and for assumptions (B1), assume that $s_0 < \eps$. Then, there exists $a^* > 0$ such that for any $a < a^*$, and any shock $(u_-, u_+)$, where $u_+=S^1_{u_-}(s)$ verifying \eqref{entropic} \eqref{lax} such that $u_- \in \Pi_a$
\begin{align*}
 D_{RH}(u_-,u_+) \leq 0.
\end{align*}
\end{proposition}
We end this section with a lemma that coarsely describes the geometry of $\Pi_a$.
\begin{lemma}[Lemma 3.2 in \cite{MR4176349}]\label{piaasymptotics}
There exists constants $a^*, C > 0$ such that $diam(\Pi_a) \leq C\sqrt{a}$ for $0 < a < a^*$.
\end{lemma}
For the proof, we refer to \cite{MR4176349}.

\section{Proof of Theorems \ref{main1} \& \ref{main2}}\label{theorempf}
In this section, we construct the shift $h$ and show how Propositions \ref{mainprop1} and \ref{mainprop2} imply Theorems \ref{main1} and \ref{main2}. The proofs of both theorems are exactly the same, so for brevity we will only show how Proposition \ref{mainprop2} implies Theorem \ref{main2}. Our proof of this step is exactly the same as in \cite{MR4667839}; our novel work is relegated to the proofs of Proposition \ref{mainprop1} and \ref{mainprop2}. We break the proof into four steps. \newline
\textbf{\underline{Step 1: Fixing $a^*$ and $\eps$}}
\par Firstly fix the shock $(u_L, u_R)$ and the parameters $a^*$ and $\eps$ sufficiently small so that Propositions  \ref{mainprop1} and \ref{mainprop2} are verified. Then, fix $0 < a < a^*$, which fixes the set $\Pi_a$. \newline
\textbf{\underline{Step 2: Control of $q(\cdot;\cdot)$ by $\eta(\cdot|\cdot)$}}
\par In this step, we provide a crucial lemma that will control the component of $D_{RH}$ involving $q(\cdot;\cdot)$ by the portion involving $\eta(\cdot|\cdot)$ for states outside of $\Pi_a$.
\begin{lemma}[Lemma 4.1 in \cite{MR4667839}]\label{qcontrol}
There exists a constant $C^* > 0$ such that
\begin{align*}
|\frac{1}{a}q(u;u_L)-q(u;u_R)| \leq C^*|\frac{1}{a}\eta(u|u_L)-\eta(u|u_R)|, &\text{ for } u \in \Pi_a^c \\
&\text{ s.t. } \frac{1}{a}q(u;u_L)-q(u;u_R) \leq 0.
\end{align*}
\end{lemma}
For the proof, we refer to \cite{MR4667839}. We note that in their proof of Lemma 4.1, they crucially use their Proposition 2.1, but our Proposition \ref{mainprop1} plays the same role and so the proof proceeds. \newline
\textbf{\underline{Step 3: Construction of the shift}}
\par Consider the set $\Pi_a^c$, which is an open set. Therefore, $-\mathds{1}_{\Pi_a^c}$ is upper semi-continuous. Consider the velocity
\begin{align}\label{velocity}
V(u):=\lambda_1(u)-(C^*+2L)\mathds{1}_{\Pi_a^c}(u), 
\end{align}
where $L:=\sup_{u \in \V}|\lambda_1(u)|$ and $C^*$ is from Lemma \ref{qcontrol}. We solve the following ODE with discontinuous right-hand side in the sense of Filipov
\begin{align}\label{shift}
\begin{cases}
\dot{h}(t)=V(u(h(t),t)), \\ 
h(0)=0.
\end{cases}
\end{align}
The existence of such a shift satisfying \eqref{shift} is given by the following lemma. For a proof, we refer to \cite{MR2807139}.
\begin{lemma}\label{shiftexistence}
Let $V:\V \to \R$ be bounded and upper semi-continuous on $\V$, and continuous on $U$ an open, full-measure subset of $\V$. Let $u \in \mathcal{S}_{weak}$. Then, there exists a Lipschitz function $h:[0,\infty) \to \R$ such that
\begin{align*}
&V_{min}(t) \leq \dot{h}(t) \leq V_{max}(t), \\
&h(0)=0, \\
&Lip[h] \leq ||V||_\infty, \\
\end{align*}
for almost every $t$, where $u_\pm:=u(h(t)\pm, t)$ and 
\begin{align*}
V_max(t)=\max(V(u_+), V(u_-), \indent \text{ and } \indent V_{min}(t)=\begin{cases}
    \min(V(u_+), V(u_-)),  & u_+, u_- \in U, \\
    -||V||_\infty & \text{otherwise}.
\end{cases}
\end{align*}
Furthermore, for almost every $t$
\begin{align*}
f(u_+)-f(u_-)&=\dot{h}(u_+-u_-), \\
q(u_+)-q(u_-) &\leq \dot{h}(\eta(u_+)-\eta(u_-)), 
\end{align*}
i.e. for almost every $t$ either $(u_-, u_+)$ is a shock verifying \eqref{entropic} or $u_-=u_+$. 
\end{lemma}
\begin{remark}\label{laxremark}
Note that by the definition of $\mathcal{S}_{weak}$, if $u_- \neq u_+$ then we further have that the shock $(u_-,u_+)$ verifies the Lax condition \eqref{lax}.
\end{remark}
\textbf{\underline{Step 4: Proof of the contraction}}
\par Finally, we show the contraction of the dissipation functional. We split into the following four cases (which allow for $u_+=u_-$)
\begin{align*}
\text{Case 1. } & u_-, u_+ \in \Pi_a^c, \\ 
\text{Case 2. } & u_- \in \Pi_a^c, u_+ \in \Pi_a, \\
\text{Case 3. } & u_- \in \Pi_a, u_+ \in \Pi_a^c, \\
\text{Case 4. } & u_-,u_+ \in \Pi_a.
\end{align*}
As we only need to prove the contraction for almost every $t$, by Lemma \ref{shiftexistence}, we may consider only times for which either $(u_-, u_+, \dot{h}(t))$ is a shock verifying \eqref{entropic}\eqref{lax} (see Remark \ref{laxremark}) or $u_+=u_-$. 
\newline \underline{Case 1:} In this case, by Lemma \ref{shiftexistence} and \eqref{velocity}, we see $\dot{h}(t) < \inf_{u \in \V}\lambda_1(u)$. If $u_- \neq u_+$, then this is a contradiction to \eqref{lax}. So necessarily $u_-=u_+$. Denote $u=u_-=u_+$. Then, as $u \not \in \Pi_a$, $\eta(u|u_L)-a\eta(u|u_R) > 0$. So, using $\dot{h}(t) \leq -C^*$:
\begin{align*}
\dot{h(t)}(\eta(u|u_L)-a\eta(u|u_R))-q(u;u_L)+aq(u;u_R)&=a(-\frac{1}{a}q(u;u_L)+q(u;u_R) \\
&+\dot{h}(t)(\frac{1}{a}\eta(u|u_L)-\eta(u|u_R))) \\
&\leq a\bigg(-\frac{1}{a}q(u;u_L)+q(u;u_R) \\
&-C^*(\frac{1}{a}\eta(u|u_L)-\eta(u|u_R))\bigg).
\end{align*}
If $q(u;u_L)-aq(u;u_R) \geq 0$, then the right-hand side is non-positive. If $q(u;u_L)-aq(u;u_R) \leq 0$, Lemma \ref{qcontrol} grants that the right-hand side is non-positive. 
\newline \underline{Case 2:} In this case, we must have $u_- \neq u_+$. Further:
\[
\dot{h}(t) \in [\lambda_1(u_-)-(C^*+2L), \lambda_1(u_+)] \implies \dot{h}(t) \leq \lambda_1(u_+).
\]
This contradicts \eqref{lax}, so we conclude that this case may not happen.
\newline \underline{Case 3:} In this case necessarily $u_- \neq u_+$. Consequently, we have for the velocity
\[
\dot{h}(t) \in [\lambda_1(u_+)-(C^*+2L), \lambda_1(u_-)] \implies \dot{h}(t) \leq \lambda_1(u_-).
\]
Thus, assumption (B1) (g) grants that $u_+=S^1_{u_-}(s)$ for some $s$. Proposition \ref{mainprop2} gives the result immediately. 
\newline \underline{Case 4:} If $u_- \neq u_+$, then following the logic in case 3 and using Proposition \ref{mainprop2} gives the result. If $u_-=u_+$, we use Proposition \ref{mainprop1} instead. This completes the proof of Theorem \ref{main2}.

\section{Proof of Propositions \ref{mainprop1} \& \ref{mainprop2}}\label{proppf}
In this section, we will prove Propositions \ref{mainprop1} and \ref{mainprop2}. The macroscopic structure of the proof is the same for both sets of assumptions (A1) and (B1), but there are some slight differences in some of the details and proofs. We will go through the proofs simultaneously, highlighting the differences when they arise. The key differences are mostly in the proof of Lemma \ref{dissipationformula2}. The case (B1) is slightly more delicate than the case (A1) as in the former case, shocks verifying \eqref{lax} may be obtained by moving both forward and backwards along the shock curve, while in the latter case they may only be obtained in the positive part of the shock curve. For convenience, we will use the notation $\sigma_u(s):=\sigma^1_u(s)$. We start with a lemma that gives an explicit formula for entropy lost from traveling along the shock curve.
\begin{lemma}[Lemma 3 in \cite{MR3537479}]\label{dissipationformula}
Let $(u_-, u_+)$ be a shock with speed $\sigma_\pm$ verifying \eqref{entropic}. Then, for any $v \in \V$
\begin{align*}
q(u_+;v)-\sigma_\pm \eta(u_+|v) \leq q(u_-;v)-\sigma \eta(u_-|v).
\end{align*}
Further, if $u_- \in B$, as in Assumption (A1) or (B1), and there exists $s$ such that $u_+=S^1_{u_-}(s)$, then
\begin{align*}
q(u_+;v)-\sigma_\pm \eta(u_+|v)-q(u_-;v)+\sigma_\pm \eta(u_-|v)=\int_0^s\sigma'_{u_-}(t)\eta(u_-|S^1_{u_-}(t))\diff t.
\end{align*}
\end{lemma}
For the proof, we refer to \cite{MR3537479}. The next lemma shows that traversing a certain portion of the shock curve in the convex-concave case results in net negative entropy dissipation.
\begin{lemma}\label{netneg}
Assume that $\lambda_1$ is convex-concave, that is it satisfies assumptions (B1). Then, for any $u_- \in B$, we have
\begin{align*}
\int_0^{s_{u_-}^{-\natural}}\sigma'_{u_-}(t)\eta(u_-|S^1_{u_-}(t))\diff t \leq 0.
\end{align*}
\end{lemma}
\begin{proof}
Let $u_+=(u_-, S^1_{u_-}(s_{u_-}^{-\natural}))$. By assumption (B1)(c), the shock $(u_-, u_+)$ verifies the Lax condition. For convex-concave characteristic fields, the Lax and Liu conditions are equivalent (cf. Lemma 2.4 in \cite{MR1759195}), so we have that $(u_-, u_+)$ verifies \eqref{entropic}. Consequently, applying Lemma \ref{dissipationformula}, we see
\begin{align*}
\int_0^{s_{u_-}^{-\natural}}\sigma'_{u_-}(t)\eta(u_-|S^1_{u_-}(t))\diff t=q(u_+;v)-\sigma_\pm \eta(u_+|v)-q(u_-;v)+\sigma_{\pm} \eta(u_-|v) \leq 0.
\end{align*}
\end{proof}
The next lemma gives a formula and crucial estimates for net entropy dissipation between two points along the shock curve. Recall that $s_0$ is the fixed parameter designating the shock size $u_R=S^1_{u_L}(s_0)$.
\begin{lemma}\label{dissipationformula2}
There exists $a^*, \eps, \delta, \kappa > 0$ such that the following is true. For assumptions (A1), assume that $s_0 < s_{u_L}$ and for assumptions (B1), assume that $s_0 < \eps$. For any $u \in \Pi_a$, and $s$ such that $S^1_{u}(s)$ obeys the Lax condition \eqref{lax}, we have
\begin{align*}
q(S^1_{u}(s);S^1_u(s_0))-\sigma_u(s)\eta(S^1_u(s)|S^1_u(s_0))=\int_{s_0}^s\sigma_u'(t)(\eta(u|S_u(t))-\eta(u|S_u(s_0)))\diff t.
\end{align*}
Further, we have estimates
\begin{align*}
q(S^1_{u}(s);S^1_u(s_0))-\sigma_u(s)\eta(S^1_u(s)|S^1_u(s_0)) \leq -\kappa |\sigma_u(s)-\sigma_u(s_0)|^2, \indent |s-s_0| \leq \delta, \\
q(S^1_{u}(s);S^1_u(s_0))-\sigma_u(s)\eta(S^1_u(s)|S^1_u(s_0)) \leq -\kappa |\sigma_u(s)-\sigma_u(s_0)|, \indent |s-s_0| \geq \delta.
\end{align*}
\end{lemma}
\begin{proof}
We handle the assumptions (A1) case first. Firstly, as $s_0 < s_{u_L}$, we have $\sigma'_{u_L}(s_0) < 0$ by assumption (A1) (e). So, taking $a^*$ sufficiently small, by Lemma \ref{piaasymptotics} we may have $\sigma'_u(s_0) < 0$ for all $u \in \Pi_a$. So, we obtain the result in the same way as the proof of Lemma 4 in \cite{MR3537479}.
\par Now we handle the assumptions (B1) case. Firstly, we may take $\delta < \frac{s_0}{2}$ so that if $|s-s_0| \leq \delta$, then $s > 0$. So, just as above, following the proof of Lemma 4 in \cite{MR3537479}, there exists $\kappa_1$ such that
\begin{align}\label{intermediate}
&q(S^1_{u}(s);S^1_u(s_0))-\sigma_u(s)\eta(S^1_u(s)|S^1_u(s_0)) \leq -\kappa_1 |\sigma_u(s)-\sigma_u(s_0)|^2, |s-s_0| \leq \delta, \\ \notag 
&q(S^1_{u}(s);S^1_u(s_0))-\sigma_u(s)\eta(S^1_u(s)|S^1_u(s_0)) \leq -\kappa_1 |\sigma_u(s)-\sigma_u(s_0)|,|s-s_0| \geq \delta, s \geq 0.
\end{align}
It remains to remove the condition $s \geq 0$ from the second estimate in \eqref{intermediate}. By assumption (B1) (d), we need only handle the case $s \leq s_u^{-\natural}$. We compute using the dissipation formula
\begin{align*}
q(S^1_{u}(s);S^1_u(s_0))-\sigma_u(s)\eta(S^1_u(s)|S^1_u(s_0))&=\int_{s_0}^s\sigma_u'(t)(\eta(u|S_u(t))-\eta(u|S_u(s_0)))\diff t \\
&= \int_0^{s_0} \sigma_u'(t)(\eta(u|S^1_u(s_0))-\eta(u|S^1_u(t)))\diff t \\
&+\int_s^0 \sigma_u'(t)\eta(u|S^1_u(s_0))\diff t \\
&+\int_0^{s_{u}^{-\natural}}\sigma'_{u}(t)\eta(u_-|S^1_{u}(t))\diff t \\
&+\int_{s_{u}^{-\natural}}^s \sigma_u'(t)\eta(u|S^1_u(t))\diff t.
\end{align*}
Using \eqref{intermediate} and Lemma \ref{netneg}, we have
\begin{align}\label{stl}
q(S^1_{u}(s);S^1_u(s_0))-\sigma_u(s)\eta(S^1_u(s)|S^1_u(s_0)) &\leq -\kappa_1|\lambda_1(u)-\sigma_u(s_0)| \\ \notag 
&+(\lambda_1(u)-\sigma_u(s))\eta(u|S^1_u(s_0)) \\
&+\int_{s_{u}^{-\natural}}^s \sigma_u'(t)\eta(u|S^1_u(t))\diff t. \notag 
\end{align}
Now, choose $\eps > 0$ such that $\eta(u_L|S^1_{u_L}(\eps))=\eta(u_L|S^1_{u_L}(s_{u_L}^{-\natural}))$ (note that this is uniquely defined by assumption (B1) (g)). Then, as $s_0 < \eps$, we have $\eta(u_L|S^1_{u_L}(s_0))<\eta(u_L|S^1_{u_L}(s_{u_L}^{-\natural}))$. By assumption (B1) (e), both sides of the inequality are continuous and we may take $a^*$ sufficiently small (cf. Lemma \ref{piaasymptotics}) so that for some $\nu > 0$ 
\begin{align}\label{destimate}
\eta(u|S^1_{u}(s_{u}^{-\natural}))-\eta(u|S^1_{u}(s_0)) > \nu, \indent u \in \Pi_a.
\end{align}
Plugging \eqref{destimate} into \eqref{stl}, we obtain
\begin{align*}
q(S^1_{u}(s);S^1_u(s_0))-\sigma_u(s)\eta(S^1_u(s)|S^1_u(s_0)) &\leq -\kappa_1|\lambda_1(u)-\sigma_u(s_0)| \\ 
&+(\lambda_1(u)-\sigma_u(s))\eta(u|S^1_u(s_0)) \\
&+(\sigma_u(s)-\sigma_u(S^1_u(s_u^{-\natural}))(\eta(u|S^1_u(s_0))+\nu).
\end{align*}
As $\sigma_u(S^1_u(s_u^{-\natural}))=\lambda_1(u)$ (see Lemma 2.3 in \cite{MR1759195}), we see that the third term on the right-hand side will absorb the second and we have
\begin{align*}
q(S^1_{u}(s);S^1_u(s_0))-\sigma_u(s)\eta(S^1_u(s)|S^1_u(s_0)) &\leq -\kappa_1|\lambda_1(u)-\sigma_u(s_0)|-\nu |\lambda_1(u)-\sigma_u(s)| \\
&\leq -\kappa |\sigma_u(s_0)-\sigma_u(s)|,
\end{align*}
where $\kappa=\min(\kappa_1, \nu)$. 
\end{proof}
The next lemma give asymptotics on some of the terms in $D_{RH}$ for states sufficiently close to $u_L$. Let $\sigma_{LR}:=\sigma_{u_L}(s_0)$.
\begin{lemma}[Lemma 4.3 in \cite{MR3519973}]\label{closetoul}
There exist $a^*, \beta >0$ and $\sigma_0 \in (\sigma_{LR}, \lambda_1(u_L))$ such that for any $0 < a < a^*$, and any $u \in \Pi_a$
\begin{align*}
    \sigma_0 &\leq \lambda_1(u), \\
    -q(u;u_L)+\sigma_0 \eta(u|u_L) &\leq -\beta \eta(u|u_L), \\
    q(u;u_R)-\sigma_0 \eta(u|u_R) &\leq -\beta \eta(u|u_R).
\end{align*}
\end{lemma}
For the proof, we refer to \cite{MR3519973}. 
\newline \textbf{\underline{Proof of Proposition \ref{mainprop1} and \ref{mainprop2}}}.
\par We will show only the proof for the case of assumptions (B1) because the (A1) case is simpler and follows only considering the scenario $s \geq 0$ in the (B1) case. Recall that we have $u_+=S^1_{u_-}(s)$ and $\sigma_\pm=\sigma_{u_-}(s)$. It suffices to prove Proposition \ref{mainprop2} as Proposition \ref{mainprop1} follows as the limit case $s=0$. Firstly, by assumption (B1) (f), there exists a unique parameter $s_0^* < s_{u_L}^{-\natural}$ such that
\begin{align*}
\sigma_{LR}=\sigma_{u_L}(s_0^*).
\end{align*}
To start, take $a^*$ sufficiently small so that Lemmas \ref{dissipationformula2} and \ref{closetoul} are satisfied. By taking $a^*$ smaller if needed and $\sigma_0$ closer to $\lambda_1(u_L)$, we may further have for any $u \in \Pi_a$
\begin{align}\label{casework}
    &\sigma_u(s) \leq \sigma_0 \text{ for } s \geq \frac{s_0}{2} \text{ or } s \leq \frac{s_u^{-\natural}-s_0^*}{2}, \\ \notag 
    &\lambda_1(u)-\sigma_0 < \frac{\kappa \Theta}{4C_1},
\end{align}
where $\kappa$ appears in Lemma \ref{dissipationformula2}, $\Theta$ appears in \eqref{theta}, and $C_1$ appears in \eqref{c1}. This is possible as $(s,u) \mapsto \sigma_u(s)$ is $C^1$, $\dv{}{s}\sigma_u(s)\neq 0$ for $s=0,s_u^{-\natural}$, and $\sigma_u(0)=\sigma_u(s_u^{-\natural})=\lambda_1(u)$ by assumptions (B1) (e). Let us first show Proposition \ref{mainprop2} for $s$ such that $\sigma_u(s) \leq \sigma_0$. Via the same process as in the proof of (4.41) in \cite{MR3519973}, we may obtain
\begin{align*}
q(u_+;u_R)-\sigma_\pm \eta(u_+|u_R) &\leq q(u_+|S^1_{u_-}(s_0))-\sigma_{u_-}(s)\eta(u_+|S^1_{u_-}(s_0)) \\
&+C|u_--u_L|^2(1+|\sigma_{u_-}(s)-\sigma_{u_-}(s_0)|) \\
&+C|u_--u_L||\sigma_{u_-}(s)-\sigma_{u_-}(s_0)|.
\end{align*}
Then, by Lemma \ref{dissipationformula2}, if $|s-s_0| \leq \delta$
\begin{align}\label{q1}
q(u_+;u_R)-\sigma_\pm \eta(u_+|u_R) &\leq -\kappa |\sigma_{u_-}(s)-\sigma_{u_-}(s_0)|^2 \\ \notag 
&+C|u_--u_L|^2(1+|\sigma_{u_-}(s)-\sigma_{u_-}(s_0)|) \\ \notag 
&+C|u_--u_L||\sigma_{u_-}(s)-\sigma_{u_-}(s_0)| \\ \notag 
& \leq C(|u_--u_L|^2+|u_--u_L|^4) \leq C|u_--u_L|^2,
\end{align}
where we have used Young's inequality and $|u_--u_L| \leq \sqrt{a^*} \leq 1$. If $|s-s_0| \geq \delta$, Lemma \ref{dissipationformula2} again yields
\begin{align}
q(u_+;u_R)-\sigma_\pm \eta(u_+|u_R) \leq C|u_--u_L|^2,
\end{align}\label{q2}
where we have taken $a^*$ much smaller if needed so that $|u_--u_L|^2 \leq Ca^* \ll \kappa$. Therefore, by \eqref{q1}, \eqref{q2}, and Lemma \ref{closetoul} we have
\begin{align}
q(u_+;u_R)-\sigma_\pm \eta(u_+|u_R) \leq aC|u_--u_L|^2.
\end{align}
So, by Lemmas \ref{closetoul} and \ref{quadraticentropy}, we have
\begin{align*}
D_{RH}(u_-, u_+) &\leq aC|u_--u_L|^2-q(u|u_L)+\sigma_\pm \eta(u_-|u_L) \\
&\leq aC|u_--u_L|^2-q(u|u_L)+\sigma_0 \eta(u_-|u_L) \\
&\leq aC|u_--u_L|^2-\beta \eta(u|u_L) \\
&\leq aC|u_--u_L|^2 -\beta c^*|u_--u_L|^2.
\end{align*}
Taking $a^*$ sufficiently small gives the result in this case.
\par Next, we handle the case $\sigma_u(s) > \sigma_0$. By \eqref{casework}, we then must have either $0 \leq s < \frac{s_0}{2}$ or $\frac{s_u^{-\natural}-s_0^*}{2} < s \leq s_u^{-\natural}$. As $\delta < \frac{s_0}{2}$, this implies $|s-s_0| \geq \delta$. Thus, using Lemma \ref{dissipationformula2} and taking $a^*$ sufficiently small
\begin{align*}
q(u_+;u_R)-\pm \eta(u_+|u_R) &\leq -\kappa|\sigma_{u_-}(s)-\sigma_{u_-}(s_0)| \\
&+C|u_--u_L|^2(1+|\sigma_{u_-}(s)-\sigma_{u_-}(s_0)|) \\
&+C|u_--u_L||\sigma_{u_-}(s)-\sigma_{u_-}(s_0)| \\
&\leq -\frac{\kappa}{2}|\sigma_{u_-}(s)-\sigma_{u_-}(s_0)|+C|u_--u_L|^2.
\end{align*}
Now, define 
\begin{align}\label{theta}
\Theta:=\inf_{u \in \Pi_a}(|\sigma_{u}(\frac{s_0}{2})-\sigma_u(s_0)|, |\sigma_u(\frac{s_u^{-\natural}-s_0^*}{2})-\sigma_u(s_0)|).
\end{align} 
By compactness of $\Pi_a$ and the continuity in assumption (B1) (e), we have $\Theta > 0$. Using assumption (B1) (e) and the smallness of $a^*$, we obtain
\begin{align}\label{thetabd}
q(u_+;u_R)-\sigma_\pm \eta(u_+|u_R) &\leq -\frac{\kappa \Theta}{4}.
\end{align}
Next, using the definition of $\Pi_a$ and Lemma \ref{closetoul} gives
\begin{align}\label{c1}
-q(u_-|u_L)+\sigma_\pm \eta(u_-|u_L) &\leq -q(u_-|u_L)+\lambda_1(u_-) \eta(u_-|u_L) \\ \notag 
&=-q(u_-|u_L)+\sigma_0 \eta(u_-|u_L)+(\lambda_1(u_-)-\sigma_0)\eta(u_-|u_L) \\ \notag 
& \leq (\lambda_1(u_-)-\sigma_0)\eta(u_-|u_L) \\ \notag 
&\leq a(\lambda_1(u)-\sigma_0)\eta(u_-|u_R) \\ \notag 
& \leq aC_1(\lambda_1(u)-\sigma_0).
\end{align}
Thus, by \eqref{thetabd} and \eqref{c1} we have
\begin{align*}
D_{RH}(u_-, u_+) \leq a\left(-\frac{\kappa \Theta}{4}+C_1(\lambda_1(u)-\sigma_0)\right).
\end{align*}
Using the bound \eqref{casework} gives the result. 

\section{Applications}\label{applications}
\subsection{Scalar conservation laws}\label{scl}
Our theory applies to the scalar conservation law:
\begin{align}\label{scalar}
u_t+(f(u))_x=0,
\end{align}
where $f: \R \to \R$ is the flux function which is either concave-convex or convex-concave in the sense of LeFloch \cite{MR1927887}. The case of concave-convex fluxes was treated in previous work of the author \cite{concaveconvex} (in fact a stronger result than Theorem \ref{main1} was proven in that special setting), so we will only address convex-concave fluxes here, i.e. fluxes which satisfy:
\begin{equation}\label{convconc}
\begin{cases}
uf''(u) < 0, & (u \neq 0), \\
f'''(0) \neq 0, \\
\lim_{|u| \to +\infty}f'(u)=+\infty.
\end{cases}
\end{equation}
The principal example is $f(u)=-u^3$. Here, the degenerate manifold is $\mathcal{M}_1=\{0\}$. Clearly, one may immediately apply Theorem \ref{main2} with any arbitrary strictly convex entropy to obtain $L^2$-stability for shocks of moderate strength. However, it is our aim in this section to obtain stability for shocks without a restriction on the size of the shock. For concreteness, fix $u_L > 0$. Then, the part of the shock curve $S^1_{u_L}(s)$ with $s > 0$ is $\{u|u > u_L\}$. We can show the following:
\begin{lemma}\label{main3}
Fix a shock $u_L, u_R$ with $u_R=S^1_{u_L}(s_0)$ for some $s_0 > 0$. Then, there exists an entropy $\overline{\eta}$ such that $\overline{\eta}$ verifies the assumptions (B1) and the parameter $\eps$ in Theorem \ref{main2} verifies $\eps > s_0$.
\end{lemma} 
\begin{proof}
Recall from the proof of Lemma \ref{dissipationformula2} that $\eps$ is such that $\eta(u_L|S_{u_L}^1(\eps))=\eta(u_L|S_{u_L}^1(s_{u_L}^{-\natural}))$. So, we simply need to construct a strictly convex function $\overline{\eta}$ verifying the following properties (restricting to $u > 0$ for convenience)
\begin{enumerate}
    \item $\partial_v \overline{\eta}(u|v) > 0$ for $v >u$,
    \item $\partial_v \overline{\eta}(u|v) < 0$ for $v < u$,
    \item $\overline{\eta}(u_L|u_R) < \overline{\eta}(u_L|S_{u_L}^1(s_{u_L}^{-\natural}))$.
\end{enumerate}
We define $\overline{\eta}$ via its second derivative. Define
\begin{align*}
    \overline{\eta}''(x)=
    \begin{cases}
    1 & x \leq 0, \\
    \frac{(\delta-1)}{u_L}x+1 & 0 \leq x \leq u_L, \\
    \delta & x \geq u_L,
    \end{cases}
\end{align*}
for some parameter $\delta > 0$ to be determined. Evidently, $\overline{\eta}''$ is continuous and bounded between $\delta$ and $1$. Then, we compute
\begin{align*}
\partial_v \overline{\eta}(u|v)=-\overline{\eta}''(v)(u-v).
\end{align*}
Clearly we have $\partial_v \overline{\eta}(u|v) > 0$ for $v >u$ and $\partial_v \overline{\eta}(u|v) < 0$ for $v < u$. Finally, we compute:
\begin{align*}
\overline{\eta}(u_L|u_R)=\int_{u_L}^{u_R}\partial_v \overline{\eta}(u_L|v)\diff v=-\delta \int_{u_L}^{u_R}(u_L-v)\diff v=\frac{\delta}{2}(u_L-u_R)^2.
\end{align*}
Similarly
\begin{align*}
\overline{\eta}(u_L|S_{u_L}^1(s_{u_L}^{-\natural})) \geq \overline{\eta}(0|S^1_{u_L}(s_{u_L}^{-\natural})) \geq \frac{1}{2}|S_{u_L}^1(s_{u_L}^{-\natural})|^2.
\end{align*}
Taking $\delta$ sufficiently small grants the result.
\end{proof}
Consequently, applying Theorem \ref{main2} with the entropy $\overline{\eta}$ grants $L^2$-stability for the shock $(u_L, u_R)$. The figure below shows, for fixed $u_L > 0$, which shocks are admissible via the Lax condition \eqref{lax} and which shocks we obtain stability for.  We note that the shocks with $u_R$ in the yellow region are stable in $L^1$ via the theory of Kru\v{z}kov \cite{MR267257}. However, we are unable to obtain $L^2$-stability for these shocks. 
\begin{figure}[H]
    \centering
    \includegraphics{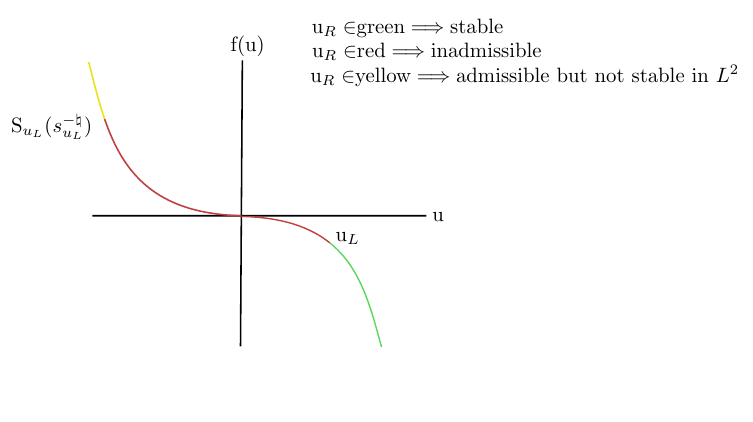}
    \caption{Stability in the scalar setting}
\end{figure}

\subsection{Nonlinear elastodynamics}\label{elastodynamics}
Our theory also applies to the system of nonlinear elastodynamics \eqref{elastsystem}:
\begin{align*}
\begin{cases}
w_t-v_x=0, & (t,x) \in \R^+ \times \R, \\
v_t+p(w)_x=0.
\end{cases}
\end{align*}
The functions $v$ and $w$ represent the velocity and deformation gradient respectively. The stress-strain law $p$ is assumed to have the form:
\begin{align*}
p(w)=-w^3-mw, \indent m>0.
\end{align*}
Define $c(w)=\sqrt{3w^2+m}$ and $u=(w,v)^T$. Then, system $\eqref{elastsystem}$ has eigenvalues and eigenvectors
\begin{align}\label{evalvec}
\lambda_{1,2}(u)=\mp c(w), \indent r_{1,2}(u)=(\pm 1, c(w))^T,
\end{align}
and the system has the strictly convex entropy
\begin{align*}
\eta(u)=\frac{v^2}{2}+\frac{1}{4}w^4+\frac{mw^2}{2}.
\end{align*}
We compute
\begin{align*}
\lambda_1'(u)\cdot r_1(u)=\frac{-3w}{c(w)}=\lambda_2'(u) \cdot r_2(u).
\end{align*}
This implies that $\mathcal{M}_1=\mathcal{M}_2=\{(w,v)|w=0\}$. Another elementary computation yields
\begin{align*}
(\lambda_1'(u)\cdot r_1(u))' \cdot r_1(u)=\frac{-3c(w)+3wc'(w)}{c(w)^2} < 0,
\end{align*}
so $\lambda_1$ is convex-concave. Similarly, we find that $\lambda_2$ is concave-convex. The shock curves, parameterized by $w$ instead of arc-length for convenience, are
\begin{align*}
&S^1_{u_L}(s)=(w_L+s,v) \text{ such that } v=v_L+\sqrt{\frac{-(p(s+w_L)-p(w_L))}{s}}(s), \\
&S^2_{u_R}(s)=(w_R+s,v) \text{ such that } v=v_R-\sqrt{\frac{-(p(s+w_R)-p(w_R))}{s}}(s), \\
\end{align*}
with corresponding speeds
\begin{align*}
&\sigma^{1}_{u_L}(s)=-\sqrt{\frac{-(p(s+w_L)-p(w_L))}{s}}, \\
&\sigma^2_{u_R}(s)=\sqrt{\frac{-(p(s+w_R)-p(w_R))}{s}}.
\end{align*}
It is tedious but straightforward to check that the assumptions (B1) are satisfied with $s_{u_L}^{\natural}=\frac{-3w_L}{2}$ and $s_{u_L}^{-\natural}=-3w_L$ (similarly for the $2$-shocks). Thus, we may apply Theorem \ref{main2} to obtain $L^2$-stability for both $1$- and $2$-shocks along the positive part of the shock curve of moderate strength (Theorem \ref{main4}).
The figures below show which shocks are admissible via the Lax condition \eqref{lax} and which shocks we obtain stability for.
 \begin{figure}[H]
     \centering
     \includegraphics{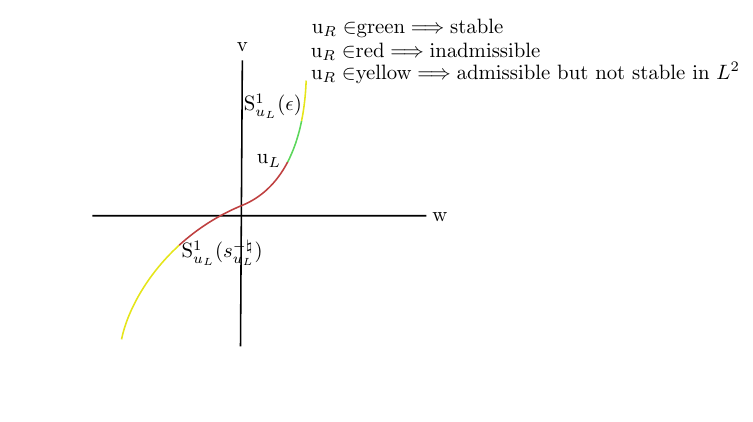}
     \caption{Stability for $1$-shocks}
 \end{figure}
 
 \begin{figure}[H]
     \centering
     \includegraphics{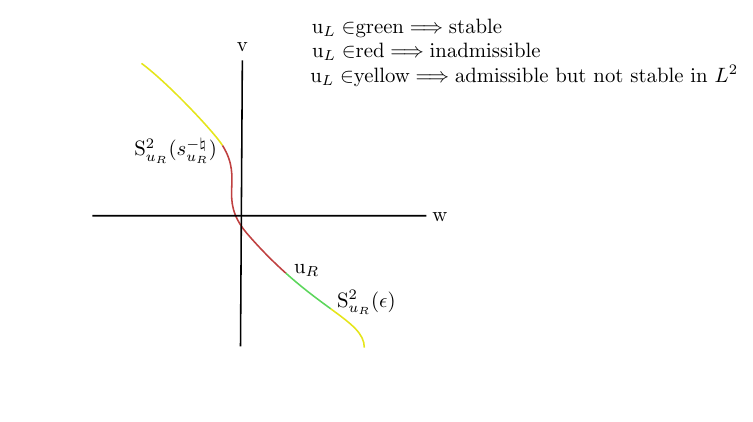}
     \caption{Stability for $2$-shocks}
 \end{figure}
Let us provide some commentary for these results. If, in Figure 2, $w_L$ is close to $0$, then the point $S^1_{u_L}(s^{-\natural}_{u_L})$ corresponding to the barrier between the yellow and red portions will also have $w$-coordinate close to $0$. Then, the small-BV $L^1$-stability theory of Ancona \& Marson \cite{MR2044745} gives $L^1$-stability for some of the shocks with $u_R=S^1_{u_L}(s)$ for $s \leq s^{-\natural}_{u_L}$ (but not all of them due to the small-BV restriction) that is not captured by Theorem \ref{main4}. 
\par However, if $w_L$ is not close to $0$, then the point $S^1_{u_L}(s^{-\natural}_{u_L})$ is also far from $0$. This implies that the parameter $\eps$ in Theorem \ref{main4} is not small either. Then, Theorem \ref{main4} gives $L^2$-stability for shocks $(u_L, S^1_{u_L}(s_0))$ for $s_0 \approx \eps$ that may not be captured by the small-BV $L^1$-stability theory.
\par Finally, for the special case of the system \eqref{elastsystem}, we conjecture that the moderate shock size condition for shocks along the positive part of the shock curve may be removed (as was done for the scalar case in Section \ref{scl}), but this remains an open question.

\bibliographystyle{plain}
\bibliography{refs}
\end{document}